\documentclass{amsart}
\usepackage{amssymb,enumerate}
\usepackage{amsmath}
\usepackage{bbm}           
\usepackage{bm}
\usepackage{eso-pic,graphicx}
\usepackage{tikz}
\usepackage{comment}

\usepackage[colorlinks=true, pdfstartview=FitV, linkcolor=blue, citecolor=blue, urlcolor=blue]{hyperref}

\makeatletter
\def\Ddots{\mathinner{\mkern1mu\raise\p@
\vbox{\kern7\p@\hbox{.}}\mkern2mu
\raise4\p@\hbox{.}\mkern2mu\raise7\p@\hbox{.}\mkern1mu}}
\makeatother

\def\Xint#1{\mathchoice
{\XXint\displaystyle\textstyle{#1}}%
{\XXint\textstyle\scriptstyle{#1}}%
{\XXint\scriptstyle\scriptscriptstyle{#1}}%
{\XXint\scriptscriptstyle\scriptscriptstyle{#1}}%
\!\int}
\def\XXint#1#2#3{{\setbox0=\hbox{$#1{#2#3}{\int}$}
\vcenter{\hbox{$#2#3$}}\kern-.5\wd0}}

\def\dashint{\Xint-}

\newtheorem{theorem}{Theorem}[section]

\newtheorem{corollary}[theorem]{Corollary}

\theoremstyle{definition}

\newtheorem{proposition}[theorem]{Proposition}

\def\bey{\begin{eqnarray*}}
\def\eey{\end{eqnarray*}}

\DeclareMathOperator{\supp}{supp}

\newcommand{\eps}{{\varepsilon}}

\newcommand{\nR}{{\mathbb R}}

\def\({\left(}
\def\){\right)}
\def\[{\left[}
\def\]{\right]}
\def\<{\langle}
\def\>{\rangle}

\newcommand{\cz}{Calder\'on-Zygmund\ }

\begin{document}


\title[Characterizations of BMO]{Characterizations of BMO through commutators of bilinear singular integral operators}

\date{\today}

\author[Lucas Chaffee]{Lucas Chaffee}
\address{%
Department of Mathematics\\
University of Kansas\\
Lawrence, KS 66045, USA}
\email{chaffel@ku.edu}

\thanks{This work was partially supported by NSF grant DMS 1069015.}

\begin{abstract}
In this paper we characterize BMO in terms of the boundedness of commutators of various bilinear singular integral operators with pointwise multiplication. In particular, we study commutators of a wide class of bilinear operators of convolution type, including bilinear \cz operators and the bilinear fractional integral operators.
\end{abstract}

\maketitle

\section{Introduction and statements of main results}
Recall that the space of functions with bounded mean oscillation, denoted BMO, consists of all locally integrable functions, $b$, such that $$\|b\|_*:=\sup_{Q}\dashint_Q|b(x)-b_Q|~dx<\infty,$$ where $Q$ is a cube with sides parallel to the axes, and $b_Q$ is the average of $b$ over $Q$.

In the linear setting, we define the commutator of a function, $b$, with an operator, $T$, acting on a function $f$ as $$[b,T](f)(x):=b(x)T(f)(x)-T(bf)(x).$$ In \cite{CRW}, Coifman, Rochberg, and Weis showed that when $T$ is the Hilbert Transform, then the linear commutator is bounded if and only if $b\in$ BMO. Note that for $f \in L^p$ and $g \in L^{p'}$ we have
$$\<[b,T](f),g\>=\<T(f)g-fT^*(g),b\>,$$
where $T^*$ denotes the transpose of $T$. In this light, we see that the characterization of the boundedness of the commutator with BMO functions means $T(f)g-fT^*(g)$, which is clearly in $L^1$, is in fact in the Hardy space $H^1$, the pre-dual of BMO. This allowed Coifman et al. to achieve a factorization of $H^1$ in a higher dimensional setting than had previously been done. Janson and Uchiyama each extended this characterization of BMO, in \cite{SJ} and \cite{Uch} respectively, to commutators of \cz operators of convolution type with smooth homogeneous kernels, and Chanillo, \cite{Cha}, did the same for commutators of the fractional integral operator with the restriction that $n-\alpha$ be an even integer. The boundedness of commutators in the multilinear setting has been extensively studied already, as in P\'erez and Torres' \cite{PT}, Tang's \cite{T}, Lerner, Ombrosi, P\'erez, Torres, and Trujillo-Gonz\'alez's \cite{LOPTTG}, Chen and Xue's \cite{CX}, and P\'erez, Pradolini, Torres, and Trujillo-Gonz\'alez's \cite{PPTT} to name a few. However, it has been an open question until now whether they can be used to characterize BMO. In this paper we will indeed show that the characterizations of BMO can be extended to a multilinear setting. For readability we will state and prove our results only for the bilinear cases.

The bilinear commutators we will be examining will be of the following forms $$[b,T]_1(f,g)(x):=bT(f,g)(x)-T(bf,g)(x),$$ and $$[b,T]_2(f,g)(x):=bT(f,g)(x)-T(f,bg)(x),$$ where $b$ is a locally integral function and $T$ is a bilinear singular integral operator.

Before we state our first result we wish to first define an $m$-linear \cz operator, as they are important to the background work of this paper and will arise in Corollary \ref{czchar}, which is itself a main result of this paper. In order to define m-linear \cz operators, we first define the class of \cz kernels. Let $K(x,y_1,...,y_m)$ be a locally integrable function defined away from the diagonal $x=y_1=...=y_m$. If for some parameters $A$ and $\eps$, both positive, we have $$|K(y_0,y_1,...,y_m)|\leq\frac{A}{\(\sum_{k,l=0}^m|y_k-y_l|\)^{mn}}$$
and $$|K(y_0,...,y_j,...,y_m)-K(y_0,...,y_j',...,y_m)|\leq\frac{A|y_j-y_j'|^\eps}{\(\sum_{k,l=0}^m|y_k-y_l|\)^{mn+\eps}}$$
whenever $0\leq j\leq m$ and $|y_j-y_j'|\leq\frac12\max_{0\leq k\leq m}|y_j-y_k|$, then we say $K$ is an $m$-linear \cz kernel. Suppose for some $m$-linear operator, $T$, defined on $L^{p_1}\times...\times L^{p_m}$, we have
$$T(f_1,...,f_m)(x)=\int K(x,y_1,...,y_m)\prod_{j=1}^mf_j(y_j)d\vec y$$ for all $x\not\in\bigcap_{j=1}^m\supp\(f_j\),$ where $K$ is a \cz kernel. Then if $$T:L^{p_1}\times...\times L^{p_m}\to L^p,$$ for some $1<p_1,...,p_m$ satisfying  $\frac1p=\sum_{j=1}^m\frac1{p_j}$, we say $T$ is an $m$-linear \cz operator. Many basic properties of these operators were thoroughly studied by L. Grafakos and R. H. Torres in \cite{GT}. Lastly, we say that an operator is of `convolution type' if the kernel  $K(x,y,z)$ is actually of the form $K(x-y,x-z)$. Our first theorem can now be stated as follows,

\begin{theorem}\label{CZOBMO}
For $b\in L^1_{loc}(\nR^n)$, and $T$ a bilinear operator defined on $L^{p_1}\times L^{p_2}$ which can be represented as $$T(f,g)(x)=\int K(x-y,x-z)f(y)g(z)dydz$$ for all $x\not\in\supp(f)\cap\supp(g),$ where $K$ is a homogeneous kernel of degree $-2n+\alpha$, and such that on some ball, $B\subset\nR^{2n}$ we have that the Fourier series of $\frac1K$ is absolutely convergent. We then have that for $1>\frac1q=\frac1{p_1}+\frac1{p_2}-\frac\alpha n$, and for $j=1$ or $2$,
$$[b,T]_j:L^{p_1}\times L^{p_2}\to L^q\implies b\in BMO(\nR^n)$$.
\end{theorem}
It is worth noting that the condition on the Fourier coefficients of the kernel will, for example, be satisfied if $K$ is smooth, and this is the assumption that similar arguments have used in the past. For $\alpha=0$, this theorem includes the case where the operator is a bilinear \cz operator, whereas if $0<\alpha<2n$, it includes the case where it is the bilinear fractional integral operator defined by $$I_\alpha(f,g)(x):=\int\int\frac{f(y)g(z)}{(|x-y|^2+|x,z|^2)^{n-\frac\alpha2}}dydz.$$ Our proof also works in the linear case, closing a gap in the literature, since in \cite{Cha} the necessity that $b\in BMO$ for the boundedness of the commutator was only shown when $n-\alpha$ was an even integer.

\subsection{Acknowledgement} This work is a part of the author's doctoral thesis at the University of Kansas, and the author would like to thank Jarod V. Hart, Rodolfo H. Torres, and David Cruz-Uribe for several insightful conversations.

\section{Proofs of the theorems}\label{proof}
The proof of theorem \ref{CZOBMO} uses techniques applied by Janson in \cite{SJ}, modified to suit the multilinear setting and extended for kernels with different homogeneity. We note that by symmetry, it is enough to prove this for $[b,T]_1$.

\begin{proof}[Proof of Theorem \ref{CZOBMO}]
Let $B=B((y_0,z_0),\delta\sqrt{2n})\subset\nR^{2n}$, be the ball for which we can express $\frac1{K(y,z)}$ as an absolutely convergent Fourier series of the form $$\frac1{K(y,z)}=\sum_ja_je^{\nu_j\cdot(y,z)}.$$ The specific vectors, $\nu_j$, will not play a role in this proof. Note that due to the homogeneity of $K$, we can take $(y_0,z_0)$ such that $|(y_0,z_0)|>2\sqrt{n}$ and take $\delta<1$ small such that $\overline B\cap\{0\}=\emptyset$. We do not care about the specific vectors $\nu_j\in \nR^{2n}$, but we will at times express them as $\nu_j=(\nu_j^1,\nu_j^2)\in \nR^n\times\nR^n$.\\
\\Set $y_1=\delta^{-1}y_0$ and $z_1=\delta^{-1}z_0$, and note that
$$\(|y-y_1|^2+|z-z_1|^2\)^{1/2}<\sqrt{2n}\implies \(|\delta y-y_0|^2+|\delta z-z_0|^2\)^{1/2}<\delta\sqrt{2n},$$
and so for all $(y,z)$ satisfying the inequality on the left, we have $$\frac1{K(y,z)}=\frac{\delta^{-2n+\alpha}}{K(\delta y,\delta z)}=\delta^{-2n+\alpha}\sum_ja_je^{i\delta\nu_j\cdot(y,z)}.$$ Let $Q=Q(x_0,r)$ be an arbitrary cube in $\nR^n$. Set $\tilde y=x_0+ry_1$, $\tilde z=x_0+rz_1$, and take $Q'=Q(\tilde y,r)\subset\nR^n$ and $Q''=Q(\tilde z,r)\subset\nR^n$. Then for any $x\in Q$ and $y\in Q',$ we have
$$	\left|\frac{x-y}r-y_1\right|\leq\left|\frac{x-x_0}{r}\right|+\left|\frac{y-\tilde y}{r}\right|\leq\sqrt{n}.$$
The same estimate holds for $x\in Q$ and $z\in Q''$, and so we have $$\(\left|\frac{x-y}r-y_1\right|^2+	\left|\frac{x-z}r-z_1\right|^2\)^{1/2}\leq\sqrt{2n}.$$ Let $\sigma(x)=$sgn$(b(x)-b_{Q'})$. We then have the following,
\begin{align*}
\int_Q|b(x)&-b_{Q'}|~dx\\
&=\int_Q(b(x)-b_{Q'})\sigma(x)~dx\\
&=\frac1{|Q''|}\frac1{|Q'|}\int_Q\int_{Q'}\int_{Q''}(b(x)-b(y))\sigma(x)~dzdydx\\
&=r^{-2n}\int_{\nR^n}\int_{\nR^n}\int_{\nR^n}(b(x)-b(y))\frac{r^{2n-\alpha}K(x-y,x-z)}{K\(\frac{x-y}r,\frac{x-z}r\)}\\
&\quad\quad\cdot\sigma(x)\chi_{Q}(x)\chi_{Q'}(y)\chi_{Q''}(z)~dzdydx\\
&=\delta^{-2n+\alpha}r^{-\alpha}\int\int\int(b(x)-b(y))K(x-y,x-z)\sum_ja_je^{i\frac\delta r\nu_j\cdot(x-y,x-z)}\\
&\quad\quad\cdot\sigma(x)\chi_{Q}(x)\chi_{Q'}(y)\chi_{Q''}(z)~dzdydx
\end{align*}
Let
\begin{align*}
f_j(y)=e^{-i\frac\delta r\nu^1_j\cdot y}\chi_{Q'}(y)\\
g_j(z)=e^{-i\frac\delta r\nu^2_j\cdot z}\chi_{Q''}(z)\\
h_j(x)=e^{i\frac\delta r\nu_j\cdot(x,x)}\sigma(x)\chi_{Q}(x)
\end{align*}
Note that Each of the above functions has an $L^q$ norm of $|Q|^{1/q}$ for any $q\geq1$. Since $Q,\ Q',$ and $Q''$ all have side length $r$, we will have that $Q\cap Q'\cap Q''=\emptyset$ if either $|x_0-\tilde y|>r\sqrt n$ or $|x_0-\tilde z|>r\sqrt n$. By the size condition on $(y_0,z_0)$ we have that either $|y_0|>\sqrt{n}$ or $|z_0|>\sqrt{n}$. If $|y_0|>\sqrt n$, we have $$|x_0-\tilde y|=\left|x_0-x_0+r\frac{y_0}{\delta}\right|\geq r|y_0|> r\sqrt n,$$ with an identical calculation if $z_0>\sqrt n$. Therefore we have that $Q\cap Q'\cap Q''=\emptyset$ since at least one of $Q'$ and $Q''$ must be disjoint from $Q$, and for all $x,\ y,$ and $z$ in the supports of their respective characteristic functions, $(x-y,x-z)$ avoids the singularity of $K$. In particular, this means that the use of the kernel representation of $[b,T](f_j,g_j)$ is valid for all $x\in Q$. Continuing with our above calculations, we have,
\begin{align*}
\int_Q|b(x)-&b_{Q'}|~dx\\
&=\delta^{-2n+\alpha}r^{-\alpha}\sum_ja_j\int h_j(x)\int\int (b(x)-b(y))\\
&\quad\quad\cdot K(x-y,x-z)f_j(y)g_j(z)~dzdydx\\
&=\delta^{-2n+\alpha}|Q|^{-\frac{\alpha}{n}}\sum_ja_j\int h_j(x)[b,T](f_j,g_j)(x)~dx\\
&\leq\delta^{-2n+\alpha}|Q|^{-\frac{\alpha}{n}}\sum_j|a_j|\int |h_j(x)||[b,T](f_j,g_j)(x)|~dx\\
&\leq\delta^{-2n+\alpha}|Q|^{-\frac{\alpha}{n}}\sum_j|a_j|\(\int |h_j(x)|^{q'}~dx\)^{\frac1{q'}}\(\int|[b,T](f_j,g_j)(x)|^q~dx\)^{\frac1q}\\
&\leq\delta^{-2n+\alpha}|Q|^{-\frac{\alpha}{n}}\sum_j|a_j|\|h_j\|_{L^{q'}}\|[b,T]\|_{L^{p_1}\times L^{p_2}\to L^p}\|f_j\|_{L^{p_1}}\|g_j\|_{L^{p_2}}\\
&=\delta^{-2n+\alpha}\|[b,T]\|_{L^{p_1}\times L^{p_2}\to L^p}\sum_j|a_j||Q|^{\frac1{q'}}|Q|^{\frac1{p_1}}|Q|^{\frac1{p_2}}|Q|^{-\frac{\alpha}{n}}\\
&=\delta^{-2n+\alpha}|Q|\|[b,T]\|_{L^{p_1}\times L^{p_2}\to L^p}\sum_j|a_j|
\end{align*}
Recall that $\frac1{|Q|}\int_Q|b(x)-b_Q|~dx\leq \frac2{|Q|}\int_Q|f(x)-C|$ for any $C$, and so this gives us  that for any arbitrary  $Q\subset\nR^n$ we have $$\frac1{|Q|}\int_Q|b(x)-b_Q|\leq\frac2{|Q|}\int_Q|b(x)-b_{Q'}|~dx\leq2\|[b,T]\|_{L^{p_1}\times L^{p_2}\to L^p}\sum_j|a_j|.$$
Therefore $b\in$ BMO$(\nR^n)$
\end{proof}

\section{Corollaries and closing remarks}\label{closing}

In Proposition 3.1 of \cite{PT}, C. P\'erez and R. H. Torres showed that $b\in BMO$ was sufficient to show the boundedness of commutators with $m$-linear \cz\ operators, which we state in a simpler bilinear format without proof.
\begin{proposition}\label{PTthm}
If $T$ is a bilinear \cz operator and $b\in$ BMO, then $[b,T]_j:L^{p_1}\times L^{p_2}\to L^p,$ for $j=1$ or $2$ and $\frac1p=\frac1{p_1}+\frac1{p_2}$ with $1<p,p_1,p_2<\infty$.
\end{proposition}

This, combined with Theorem \ref{CZOBMO}, immediately gives us the following.

\begin{corollary}\label{czchar}
Let $b\in L^1_{loc}(\nR^n)$, and $T$ a bilinear \cz operator of convolution type with kernel, $K$, a homogeneous function of degree $-2n$Suppose that on some ball, $B$, in $\nR^{2n}$ we have that the Fourier series of $\frac1K$ is absolutely convergent. Then for $1>\frac1p=\frac1{p_1}+\frac1{p_2}$, and $j=1$ or $2$,
$$[b,T]_j:L^{p_1}\times L^{p_2}\to L^p\iff b\in BMO(\nR^n)$$.
\end{corollary}
Note that the kernel conditions in the statements of this corollary is not present in the statement of Proposition \ref{PTthm}, whereas the requirement that $T$ be a \cz operator is not needed for Theorem \ref{CZOBMO}.

For $T=I_\alpha$, the sufficiency of $b\in$BMO to conclude the boundedness of $[b,I_\alpha]_i$ was shown for a class of weights which includes the unweighted case by X. Chen and Q. Xue in \cite{CX}, in their Theorem 2.7. As before, we state without proof a particular case of this theorem which suits our needs,
\begin{proposition}\label{CXthm}
    Let $0<\alpha<2n$, and $1\leq p_1,\ p_2,$ and $q$ be such that $\frac1{p_1}+\frac1{p_2}-\frac\alpha{n}=\frac1q$. Then $$\|[b,I_\alpha]_j(f,g)\|_{L^q}\lesssim\|b\|_*\|f\|_{L^{p_1}}\|g\|_{L^{p_2}}$$
    for $j=1$ or $2$.
\end{proposition}
The kernel of $I_\alpha$ has precisely the homogeneity required by Theorem \ref{CZOBMO}, and the reciprocal of the convolution kernel of $I_\alpha$, $(|y|^2+|z|^2)^{n-\alpha/2}$, is smooth away from the origin and so its Fourier series will indeed have regions on which it is absolutely convergent. These facts give us the following result.

\begin{corollary}\label{fracBMO}
For $b\in L^1_{loc}$, $0<\alpha<2n$ and $1<p_1,p_2,$ and $q$ satisfying $$\frac1{p_1}+\frac1{p_2}-\frac\alpha n=\frac1q<1,$$ then
$$\|[b,I_\alpha]_j\|_{p_1\times p_2\to q}\approx\|b\|_*\quad \text{for } j=1\text{ or }2$$
In particular, for $j=1$ or $2$, $$[b,I_\alpha]_j:L^{p_1}\times L^{p_2}\to L^q\iff b\in\text{BMO}.$$
\end{corollary}

With regards to our main theorem, two key things should be noted. First, the proof easily generalizes to commutators with the $m$-linear operators and homogeneous kernels of degree $-mn+\alpha$. The original statements of Proposition \ref{PTthm} in \cite{PT} and Proposition \ref{CXthm} in \cite{CX} ares for $m$-linear commutators, so Corollaries \ref{czchar} and \ref{fracBMO} hold in the $m$-linear setting as well. Second, note that we only assume the commutator is bounded, not the underlying operator, and this allows $\alpha<0$.

Finally, we observe that since our proof required the use of H\"older's inequality with $q$ and $q'$, the exponent in our target space must be larger than 1. We do not know if it is possible to characterize BMO in terms of the boundedness of commutators for
$L^{p_1}\times L^{p_2}\to L^q$ for $\frac12<q<1$. This is of interest because bounds of this form have indeed been shown. In particular, in \cite{LOPTTG}, Lerner et al. showed that commutators with $m$-linear \cz operators are bounded from $\prod_{j=1}^mL^{p_j}$ to $L^p$, for any $1<p_1,...,p_m$ such that $\frac1p=\sum_{j=1}^m\frac1{p_j}$, provided that $b\in$ BMO. In \cite{T}, Tang obtained this result for commutators of vector valued multilinear \cz operators, again without the restriction that $p$ be greater than 1.

\end{document}